\documentclass[10pt]{amsart}
\usepackage{amsmath}
\usepackage{amsfonts}
\usepackage{amssymb}
\usepackage{graphicx}
\usepackage{amsthm,graphicx,color,yfonts}
\usepackage{pdfsync}
\usepackage{epstopdf}%
\usepackage{pifont}
\usepackage{bbm}
\usepackage{a4wide}

\usepackage[colorlinks=true]{hyperref}
\hypersetup{linkcolor=red,citecolor=blue,filecolor=dullmagenta,urlcolor=blue} % coloured links
\usepackage{wrapfig}
\usepackage{tikz}
\usetikzlibrary{arrows,calc,decorations.pathreplacing}
\definecolor{light-gray1}{gray}{0.90}
\definecolor{light-gray2}{gray}{0.80}
\definecolor{light-gray3}{gray}{0.60}

\newcommand{\vp}{\varphi}

\newcommand{\R} {\mathbb R}
\newcommand{\cuad}{{\sqcap\kern-.68em\sqcup}}
\newcommand{\abs}[1]{\mid #1 \mid}

\newcommand{\ve}{\varepsilon}

\newcommand{\be}{\begin{equation}}
\newcommand{\ee}{\end{equation}}

\definecolor{darkgreen}{rgb}{0.2,0.7,0.1}

\newcommand{\sech}{\mathop{\mbox{\normalfont sech}}\nolimits}

\newcommand{\px}{\partial_x}

\newcommand{\nlop}{(1-\partial_x^2)^{-1}}

\newcommand{\al}{\alpha}

\def\bm{\left( \begin{array}{cc}}
\def\endm{\end{array}\right)}
\providecommand{\abs}[1]{\left|#1 \right|}
\providecommand{\norm}[1]{\left\| #1 \right\|}

\newcommand{\ba}{\begin{equation*}}
\newcommand{\ea}{\begin{equation*}}
\newcommand{\bea}{\begin{eqnarray}}
\newcommand{\eea}{\end{eqnarray}}
\newcommand{\bee}{\begin{eqnarray*}}
\newcommand{\eee}{\end{eqnarray*}}
\newcommand{\ben}{\begin{enumerate}}
\newcommand{\een}{\end{enumerate}}

\numberwithin{equation}{section}

\newtheorem{theorem}{Theorem}[section]
\newtheorem*{theorem*}{Theorem}

\newtheorem{lemma}{Lemma}[section]

\theoremstyle{remark}
\newtheorem{remark}{Remark}[section]
\title[Decay for gBBM]{Extended decay properties for generalized BBM equations}

\author{Chulkwang Kwak}
\address{Facultad de Matem\'aticas, Pontificia Universidad Cat\'olica de Chile, Campus San Joaqu\'in. Avda. Vicu\~na Mackenna 4860, Santiago, Chile}
\email{chkwak@mat.uc.cl}
\thanks{C. K. is supported by FONDECYT Postdoctorado 2017 Proyect N$^{\circ}$ 3170067.}

\author{Claudio Mu\~noz}
\address{CNRS and Departamento de Ingenier\'{\i}a Matem\'atica and Centro
de Modelamiento Matem\'atico (UMI 2807 CNRS), Universidad de Chile, Casilla
170 Correo 3, Santiago, Chile.}
\email{cmunoz@dim.uchile.cl}
\thanks{C. M. work was partly funded by Chilean research grants FONDECYT  1150202, Fondo Basal CMM-Chile, MathAmSud EEQUADD and Millennium
Nucleus Center for Analysis of PDE NC130017.}
\thanks{Part of this work was carried out while the authors were part of the Focus Program on Nonlinear Dispersive Partial Differential Equations and Inverse Scattering (August 2017) held at Fields Institute, Canada. They would like to thank the Institute and the organizers for their warming support.}


\begin{document}

%%%%%%%%%%%%%%%%%%%%%%%%%%%%%%%%%%%%%%%%%%%%%%%%%%%%%%%%%%%%%%%%%%%%%%%%%%%%%%%%%%%%%%%%%%%%%%%%%%
\begin{abstract}
In this note we show that all small solutions of the BBM equation must decay to zero as $t\to +\infty$ in large portions of the physical space, extending previous known results, and only assuming data in the energy space. Our results also include decay on the left portion of the physical line, unlike the standard KdV dynamics.  
\end{abstract}

\maketitle

\section{Introduction and Main Results}

\subsection{Setting of the problem} In this note we shall consider nonlinear scattering and decay properties for the one-dimensional generalized Benjamin, Bona and Mahony (gBBM) equation \cite{BBM} (or regularized long wave equation) in the energy space:
\begin{equation}\label{BBM}
(1- \partial_x^2)u_t  + \left( u + u^p \right)_x =0, \quad (t,x)\in \R\times\R, \quad p=2,3,4,\ldots
\end{equation}
Here $u=u(t,x)$ is a real-valued scalar function. The original BBM equation, which is the case $p=2$ above, was originally derived by Benjamin, Bona and Mahony \cite{BBM} and Peregrine \cite{Peregrine} as a model for the uni-directional propagation of long-crested, surface water waves. It also arises mathematically as a regularized version of the KdV equation, obtained by performing the standard "Boussinesq trick''. This leads to simpler well-posedness and better dynamical properties compared with the original KdV equation. Moreover, BBM is not integrable, unlike KdV \cite{BPS,MMM}. 

\medskip

It is well-known (see \cite{BT}) that \eqref{BBM} for $p=2$ is globally well-posed in $H^s $, $s\geq 0$, and weakly ill-posed for $s<0$. As for the remaining cases $p=3,4,\ldots$, gBBM is globally well-posed in $H^1$ \cite{BBM}, thanks to the preservation of the mass and energy
\be\label{Mass}
M[u](t):=\frac12 \int \left(   u^2 + u_x^2 \right)(t,x)dx,
\ee
\be\label{Energy}
E[u](t):= \int \left( \frac12  u^2 +\frac{u^{p+1} }{p+1} \right)(t,x)dx.
\ee
Since now, we will identify $H^1$ as the standard \emph{energy space} for \eqref{BBM}.

\subsection{Main result} In this note we consider the problem of decay for small solutions to gBBM \eqref{BBM}. Let $b>0$ and $a>\frac18$ be any positive numbers,
%and
%\be\label{a}
%a:= \frac16(5-\sqrt{13}) \sim 0.232,
%\ee
and $I(t)$ be given by
\be\label{I(t)}
I(t):=\left( -\infty,  -a t \right) \cup \left( (1+b) t ,  \infty\right), \quad t>0.
\ee

\begin{theorem}\label{Thm1}
Let $u_0\in H^1$ be such that, for some $\ve=\ve(b)>0$ small, one has
\be\label{Smallness}
\|u_0\|_{H^1}< \ve.
\ee
Let $u\in C(\R, H^1)$ be the corresponding global (small) solution of \eqref{BBM} with initial data $u(t=0)=u_0$. Then, for $I(t)$ as in \eqref{I(t)}, there is strong decay to zero:
\be\label{Conclusion_0}
\lim_{t \to \infty}   \|u(t)\|_{H^1(I(t))} =0.
\ee
Additionally, one has the mild rate of decay
\be\label{Conclusion_1}
\int_{2}^\infty \!\! \int e^{-c_0 |x+ \sigma t|} \left(u^2 + u_x^2\right)(t,x)dx\, dt  \lesssim_{c_0} \ve^2,
\ee
where $\sigma $ is fixed and such that $\sigma>\frac18$ or $\sigma=-(1+b)$.
\end{theorem}

\begin{remark}
The case of decay inside the interval $((1+b) t,+\infty)$ is probably well-known in the literature, coming from arguments similar to those exposed by El-Dika and Martel in \cite{ElDika_Martel}. However, decay for the left portion $\left( -\infty,  - \frac18^+ t \right)$ seems completely new as far as we understand, and it is in strong contrast with the similar decay problem for the KdV equation on the left, which has not been rigorously proved yet. 
\end{remark}

%\begin{remark}
%We believe that the constant $a$ in \eqref{a} can be improved to $a=\frac18= 0.125$ at least, but we do not have a rigorous proof of this fact. Probably more decay on the initial data is needed to reach that endpoint.
%\end{remark}

\begin{remark}
Note that our results also consider the cases $p=2$ and $p=3$, which are difficult to attain using standard scattering techniques because of very weak linear decay estimates, and the presence of long range nonlinearities. Recall that the standard linear decay estimates for BMM are $O(t^{-1/3})$ \cite{Albert}. 
\end{remark}

\begin{remark}
Theorem \ref{Thm1} is in concordance with the existence of solitary waves for \eqref{BBM} \cite{ElDika_Martel}. Indeed, for any $c>1$,
\[
u(t,x):=(c-1)^{1/(p-1)}Q\left(\sqrt{\frac{c-1}{c}} (x-ct)\right), \quad Q(s):= \left( \frac{p+1}{2\cosh^2(\frac{p-1}{2}s)} \right)^{1/(p-1)},
\]
is a solitary wave solution of \eqref{BBM}, moving to the right with speed $c>1$. Small solitary waves in the energy space have $c\sim 1$ ($p<5$), which explains the emergence of the coefficient $b$ in \eqref{I(t)}. Also, \eqref{BBM} has solitary waves with negative speed: for $c>0$ and $p$ even,
\[
u(t,x):=-(c+1)^{1/(p-1)}Q\left(\sqrt{\frac{c+1}{c}} (x+ ct)\right),
\]
is solitary wave for \eqref{BBM}, but it is never small in the energy space. The stability problem for these solitary waves it is well-known: it was studied in \cite{Bona,Weinstein,SS,BMR}. Indeed, solitary waves are stable for $p=2,3,4,5$, and stable/unstable for $p>5$, depending on the speed $c$. See also \cite{MMM} for the study of the collision problem for $p=2$.
\end{remark}

\begin{remark}
The extension of this result to the case of perturbations of solitary waves is an interesting open problem, which will be treated elsewhere.
\end{remark}

\subsection{About the literature} Albert \cite{Albert} showed scattering in the $L^\infty$ norm for solutions of \eqref{BBM} provided $p>4$, with resulting global decay $O(t^{-1/3})$. Here the power 4 is important to close the nonlinear estimates, based in weighted Sobolev and Lebesgue spaces. Biler et. al \cite{BDH} showed decay in several space dimensions, using similar techniques. Hayashi and Naumkin \cite{HN} considered BBM with a difussion term, proving asymptotics for small solutions. Our result improves \cite{Albert,HN} in the sense that it also considers the cases $p=2$ and 3, which are not part of the standard scattering theory, and it does not requires a damping term to be valid. 

\medskip

Concerning asymptotic regimes around solitary waves, the fundamental work of Miller and Weinstein \cite{MW} showed asymptotic stability of the BBM solitary wave in exponentially weighted Sobolev spaces. El-Dika \cite{ElDika,ElDika2} proved asymptotic stability properties of the BBM solitary wave in the energy space. El-Dika and Martel \cite{ElDika_Martel} showed stability an asymptotic stability fo the sum of $N$ solitary waves. See also Mizumachi \cite{Mizu} for similar results. All these results are proved on the right of the main part of the solution itself, and no information is given on the remaining left part. Theorem \ref{Thm1} is new in the sense that it also gives information on the left portion of the space.

\subsection{About the proof} In order to prove Theorem \ref{Thm1}, we follow the ideas of the proof described in \cite{KMPP}, where decay for an $abcd$-Boussinesq system \cite{Bous,BCS1,BCS2} was considered. The main tool in \cite{KMPP} was the construction of a suitable virial functional for which the dynamics is converging to zero when integrated in time. In this paper, this construction is somehow simpler but still interesting enough, because it allows to consider two different regions of the physical space, on the left (dispersive) and on the right (soliton region), unlike KdV for which virial estimates only reach the soliton region \cite{MM,MM1,MM2}.  The virial that we use here is also partly inspired in the ones introduced in \cite{KMM1,KMM2,KMM3}, and previously in \cite{MM,MR}. See also \cite{AM1,MPP,GPR} for similar results. 

\subsection{Acknowledgments} We thank F. Rousset and M. A. Alejo for many interesting discussions on this subject and the BBM equation.

\bigskip
\section{Proof of Theorem \ref{Thm1}}\label{VIRIAL}

Let $L>0$ be large, and  $\vp=\vp(x)$ be a smooth, bounded weight function, to be chosen later. For each $t,\sigma\in\R$, we consider the following functionals (see \cite{ElDika_Martel} for similar choices):
\be\label{I}
\mathcal I(t) := \frac12\int \vp\left(\frac{x+\sigma t}{L}\right)\left(u^2 + u_x^2\right)(t,x)dx,
\ee
and
\be\label{J}
\mathcal J(t) := \int \vp\left(\frac{x+\sigma t}{L}\right)\left(\frac12 u^2 + \frac{u^{p+1}}{p+1} \right)(t,x)dx.
\ee
Clearly each functional above is well-defined for $ H^1$ functions. Using \eqref{BBM} and integration by parts, we have the following standard result (see also \cite{ElDika_Martel,KMPP} for similar computations).

\begin{lemma}\label{Virial_bous}
For any $t\in \R$, we have %($\varphi' := \varphi_x$ for simplicity)
\be\label{Virial0}
\begin{aligned}
\frac{d}{dt} \mathcal I(t) = &~  {}  \frac{\sigma}{2L} \int  \varphi' u_x^2  +\frac{1}{2L}  (\sigma-1) \int  \varphi' u^2 + \frac{1}{L}\int \varphi' u (1-\partial_x^2)^{-1} u    \\
&~ {}     -\frac1{L(p+1)} \int  \varphi' u^{p+1}     + \frac{1}{L}\int \varphi'  u (1-\partial_x^2)^{-1}\left(u^p\right) ,
\end{aligned}
\ee
and if $v:=(1-\partial_x^{2})^{-1}(u + u^p)$,
\be\label{Virial1}
\begin{aligned}
\frac{d}{dt} \mathcal J(t) = &~  {}   \frac{\sigma}{2L} \int \varphi' \left( u^2 + \frac2{p+1} u^{p+1} \right) +\frac{1}{2L} \int \varphi' (v^2 - v^2_x). %\sigma \int  \varphi' \left( \frac12u^2 + \frac13u^3\right)  + \frac{1}{2}   \int  (u +u^2)^2(1-\px^2)^{-1}\varphi'  .
\end{aligned}
\ee
\end{lemma}

\begin{proof}
{\it Proof of \eqref{Virial0}.} We compute:
\[
\begin{aligned}
\frac{d}{dt} \mathcal I(t) =&~ \frac{\sigma}{2L} \int \varphi' ( u^2 + u_x^2 ) + \int \varphi (uu_t + u_x u_{tx})\\
=&~  \frac{\sigma}{2L} \int \varphi' ( u^2 + u_x^2 ) + \int \varphi u(u_t - u_{txx}) - \frac{1}{L}\int \varphi' u u_{tx}.
\end{aligned}
\]
Replacing \eqref{BBM}, and integrating by parts, we get
\[
\begin{aligned}
\frac{d}{dt} \mathcal I(t) =&~  \frac{\sigma}{2L} \int \varphi' ( u^2 + u_x^2 ) - \int \varphi u (u+u^p)_x +  \frac{1}{L}\int (\varphi' u)_{xx} (1-\px^2)^{-1} (u+u^p) \\
= :&~ I_1 + I_2 +I_3 .
\end{aligned}
\]
$I_1$ is already done.  On the other hand, 
\[
I_2=  - \int \varphi  (u u_x+ p u^p u_x)  =  \frac{1}{L}\int \varphi' \left( \frac12u^2 + \frac{p}{p+1} u^{p+1}\right). 
\]
Finally,
\[
\begin{aligned}
I_3 = &~ {}- \frac{1}{L}\int \varphi' u (u+u^p) +  \frac{1}{L}\int \varphi' u (1-\px^2)^{-1} (u+u^p) \\
= & {} - \frac{1}{L}\int \varphi' (u^2 + u^{p+1}) +  \frac{1}{L}\int \varphi' u (1-\px^2)^{-1} (u+u^p).
\end{aligned}
\]
We conclude that
\[
\begin{aligned}
\frac{d}{dt}\mathcal I(t) = &~  \frac{\sigma}{2L} \int  \varphi' u_x^2  +\frac{1}{2L}  (\sigma-1) \int  \varphi' u^2 + \frac{1}{L}\int \varphi' u (1-\partial_x^2)^{-1} u    \\
&~ {}     -\frac1{L(p+1)} \int  \varphi' u^{p+1}    + \frac{1}{L}\int \varphi'  u (1-\partial_x^2)^{-1}\left(u^p\right).
\end{aligned}
\]
This last equality proves \eqref{Virial0}.

\medskip
\noindent
{\it Proof of \eqref{Virial1}.} We compute:
\[
\begin{aligned}
\frac{d}{dt} \mathcal J(t) =&~ \frac{\sigma}{2L} \int \varphi' \left( u^2 + \frac2{p+1} u^{p+1} \right) + \int \varphi (u + u^p)u_t \\
=&~ \frac{\sigma}{2L} \int \varphi' \left( u^2 + \frac2{p+1} u^{p+1} \right)  - \int \varphi  (u + u^p) \px (1-\partial_x^{2})^{-1}(u + u^p).
%=&~ \frac{\sigma}{2} \int \varphi' ( u^2 + \frac23 u^3 )  - \int \varphi  (1-\partial_x^{2}) (1-\partial_x^{2})^{-1} (u + u^2) \px (1-\partial_x^{2})^{-1}(u + u^2).
\end{aligned}
\]
Recall that $v=(1-\partial_x^{2})^{-1}(u + u^p)$. Then
\[
 - \int \varphi  (u + u^p) \px (1-\partial_x^{2})^{-1}(u + u^p) =  - \int \varphi   (1-\partial_x^{2}) v v_x %=  - \int \varphi  (v-v_{xx}) v_x 
 = \frac{1}{2L}\int \varphi' (v^2 - v_x^2).
\]
Therefore,
\[
\frac{d}{dt} \mathcal J(t) = \frac{\sigma}{2L} \int \varphi' \left( u^2 + \frac2{p+1} u^{p+1} \right) + \frac{1}{2L}\int \varphi' (v^2 - v_x^2).
\]
This proves \eqref{Virial1}.
\end{proof}

For $\alpha$ real number, define the modified virial
\be\label{H}
\mathcal H(t):= \mathcal H_{\al}(t):= \mathcal I(t) + \alpha \mathcal J(t).
\ee
From Lemma \ref{Virial_bous}, we get (recall that $v = (1-\partial_x^{2})^{-1}(u + u^p)$)
\be\label{dH}
\begin{aligned}
\frac{d}{dt}\mathcal H(t) =& ~{}  \frac{\sigma}{2L} \int  \varphi' u_x^2  +\frac{1}{2L}  (\sigma(1+\al)-1) \int  \varphi' u^2 +  \frac{1}{L}\int \varphi' u (1-\partial_x^2)^{-1} u  \\
&~ {}   + \frac{\al}{2L}\int \varphi' (v^2 - v_x^2)    -\frac1{L(p+1)}(\al \sigma -1) \int  \varphi' u^{p+1}     + \frac{1}{L}\int \varphi'  u (1-\partial_x^2)^{-1}\left(u^p\right)  .
\end{aligned}
\ee
Let also, for $u\in H^1$,
\begin{equation}\label{eq:fg}
f := \nlop u \in H^3.% \quad \mbox{and} \quad g := \nlop \eta.
\end{equation}
We have
\[%begin{equation}\label{eq:L2}
\int \vp' u^2 = \int\vp'\left(f^2 + 2f_x^2 + f_{xx}^2\right) - \frac{1}{L^2}\int \vp'''f^2,
\]%end{equation}
\[%begin{equation}\label{eq:H1}
\int \vp'u_x^2 = \int\vp'\left(f_x^2 + 2f_{xx}^2 + f_{xxx}^2\right) - \frac{1}{L^2}\int \vp'''f_x^2,
\]%end{equation}
and
\[%begin{equation}\label{eq:nonlocal}
\int \vp' u \nlop u = \int\vp'\left(f^2 + f_x^2\right) - \frac{1}{2L^2}\int \vp'''f^2.
\]%end{equation}
Additionally, we easily have
\[%be\label{v2}
v^2 =f^2 + 2f (1-\partial_x^{2})^{-1}(u^p) + ( (1-\partial_x^{2})^{-1}u^p)^2 ,
\]%ee
and similarly,
\be\label{v2x}
v_x =f_x^2 + 2f_x (1-\partial_x^{2})^{-1}(u^p)_x + (\partial_x (1-\partial_x^{2})^{-1}u^p)^2.
\ee
Replacing these values in \eqref{dH} and rearranging terms, we get
\be\label{dH_new}
\frac{d}{dt}\mathcal H(t) =\mathcal Q(t) + \mathcal{S}(t) + \mathcal N(t),
\ee
where
\be\label{Q}
\begin{aligned}
\mathcal Q(t) :=& ~{}  \frac{1}{2L}(1+\sigma)(1+\alpha) \int  \varphi' f^2  +\frac{1}{2L}  (\sigma(3+2\al)-\al) \int  \varphi' f_x^2  \\
& ~{} + \frac{1}{2L}((3+\alpha)\sigma-1) \int \varphi' f_{xx}^2  + \frac{\sigma}{2L}\int \varphi' f_{xxx}^2 , 
\end{aligned}
\ee
\be\label{S}
\begin{aligned}
\mathcal S(t) :=& ~{}  -\frac{1}{2L^3}  (\sigma(1+\al)-1) \int \vp'''f^2 -  \frac{\sigma}{2L^3} \int \vp'''f_x^2- \frac{1}{2L^3}\int \vp'''f^2 , 
\end{aligned}
\ee
and
\be\label{N}
\begin{aligned}
\mathcal N(t) :=& ~{}\frac{\al}{2L}\int \varphi' \left( 2f (1-\partial_x^{2})^{-1}(u^p) + ( (1-\partial_x^{2})^{-1}u^p)^2  \right)  \\
& ~{} - \frac{\al}{2L}\int \varphi' \left( 2f_x (1-\partial_x^{2})^{-1}(u^p)_x + (\partial_x (1-\partial_x^{2})^{-1}u^p)^2 \right)\\
&~ {}     -\frac1{L(p+1)}(\al \sigma -1) \int  \varphi' u^{p+1}     + \frac{1}{L}\int \varphi'  u (1-\partial_x^2)^{-1}\left(u^p\right).
\end{aligned}
\ee
Now we consider two different cases.

\medskip
\noindent
{\it Case $x>0$.} This is the simpler case. We choose $\varphi := \tanh$, $\al=0$, and $\sigma =-(1+b)<0$, for $b$ any fixed positive number. Note that $\varphi' = \sech^2 >0.$ Then   
\begin{equation}\label{Positivity +}
\begin{aligned}
\mathcal Q(t) =& ~{}  -\frac{1}{2L}b \int  \varphi' f^2  - \frac{3}{2L} (1+ b) \int  \varphi' f_x^2  - \frac{1}{2L}(4+3b) \int \varphi' f_{xx}^2  - \frac{1}{2L}(1+b)\int \varphi' f_{xxx}^2.
\end{aligned}
\end{equation}
Now we recall the following result.

\begin{lemma}[Equivalence of local $H^1$ norms, \cite{KMPP}]\label{lem:L2 comparable}
Let $f$ be as in \eqref{eq:fg}. Let $\phi$ be a smooth, bounded positive weight satisfying $|\phi''| \le \lambda \phi$ for some small but fixed $0 < \lambda \ll1$. Then, for any $a_1,a_2,a_3,a_4 > 0$, there  exist $c_1, C_1 >0$, depending on $a_j$ and $\lambda >0$, such that
\begin{equation}\label{eq:L2_est}
c_1  \int \phi \, (u^2 + u_x^2) \le \int \phi\left(a_1f^2+a_2f_x^2+a_3f_{xx}^2 +a_4 f_{xxx}^2 \right) \le C_1 \int \phi \, (u^2 + u_x^2).
\end{equation}
\end{lemma}
Thanks to this lemma, we get for this case
\be\label{Caso1}
\mathcal Q(t)  \lesssim_{b,L} - \int  \varphi' (f^2 +f_x^2 +f_{xx}^2 + f_{xxx}^2) \sim -\int\varphi' (u_x^2 + u^2). 
\ee

\medskip
\noindent
{\it Case $x<0$.} Here we need different estimates. In \eqref{Q}, we will impose
\[
\sigma= \frac18(1+ \tilde\sigma),\quad \tilde\sigma >0, \quad \hbox{and}\quad \al=1.
\]
%\[
%(1+\sigma)(1+\alpha) >0, \quad \sigma(3+2\al)>\al, \quad (3+\alpha)\sigma>1,\quad \hbox{and} \quad \sigma >0.
%\]
We choose now $\varphi := -\tanh$. Note that $\varphi' = -\sech^2 <0$. Then we have
\[
\begin{aligned}
-16L\mathcal Q(t) =& ~{}  2(9+\tilde \sigma) \int  |\varphi'| f^2  +  (-3+ 5\tilde\sigma) \int  |\varphi'| f_x^2  \\
& ~{} +4 (-1 + \tilde\sigma) \int |\varphi'| f_{xx}^2  + (1+\tilde \sigma)\int |\varphi'| f_{xxx}^2 .
\end{aligned}
\]
Define $g:= |\varphi'|^{1/2} f = \sech (\frac{x+\sigma t}{L})f$. Then we have the following easy identities
\[
\begin{aligned}
g_x =&~\sech \left(\frac{x+\sigma t}{L}\right) f_x + \frac1L(\sech)'  \left(\frac{x+\sigma t}{L}\right) f \\
=&~ \sech \left(\frac{x+\sigma t}{L}\right) f_x - \frac1L\tanh \left(\frac{x+\sigma t}{L}\right) g,
\end{aligned}
\]
\[
g_{xx} = \sech \left(\frac{x+\sigma t}{L}\right) f_{xx} + \frac2L(\sech)'  \left(\frac{x+\sigma t}{L}\right) f_x + \frac1{L^2}(\sech)''  \left(\frac{x+\sigma t}{L}\right) f, 
\]
and
\[
\begin{aligned}
g_{xxx} = &~ \sech \left(\frac{x+\sigma t}{L}\right) f_{xxx} +\frac3L (\sech)'  \left(\frac{x+\sigma t}{L}\right) f_{xx} \\
&~{} + \frac3{L^2}(\sech)''  \left(\frac{x+\sigma t}{L}\right) f_x + \frac1{L^3}(\sech)'''  \left(\frac{x+\sigma t}{L}\right) f.
\end{aligned}
\]
Therefore,
{\color{black} 
\[
g_{xx} = -\frac{1}{L^2}g - \frac{2}{L} \tanh \left(\frac{x+\sigma t}{L}\right) g_x + \sech \left(\frac{x+\sigma t}{L}\right) f_{xx}
\]
and
\[
g_{xxx} = -\frac{1}{L^3} \tanh \left(\frac{x+\sigma t}{L}\right) g -\frac{3}{L^2} g_x - \frac{3}{L} \tanh \left(\frac{x+\sigma t}{L}\right) g_{xx} + \sech \left(\frac{x+\sigma t}{L}\right) f_{xxx}.
\]
}
Consequently, for $L$ large enough,
\[
\begin{aligned}
-16L\mathcal Q(t) =& ~{}  2(9+\tilde \sigma) \int g^2  +  (-3+ 5\tilde\sigma) \int  g_x^2  \\
& ~{} +4 (-1 + \tilde\sigma) \int g_{xx}^2  + (1+\tilde \sigma)\int g_{xxx}^2 +O\left(\frac1L \int (g^2 +g_x^2 + g_{xx}^2)\right).
\end{aligned}
\]
Now we have for $g\in H^3$, 
\[
\int ( g_{xxx} -\sqrt{2}g_{xx} +3 g_x -3\sqrt{2}g)^2\geq 0.
\]
Expanding terms and integrating by parts,
\[
 \int g_{xxx}^2      -4 \int g_{xx}^2  -3 \int  g_x^2+  18 \int g^2 \geq 0.
\]
We conclude that for $\tilde \sigma>0$ fixed and $L$ large enough,
\[
\begin{aligned}
-16L\mathcal Q(t) \geq   & ~{}  \tilde \sigma  \int g^2  +  4\tilde\sigma \int  g_x^2   +3 \tilde\sigma \int g_{xx}^2  + \frac12\tilde \sigma \int g_{xxx}^2.
\end{aligned}
\]
Coming back to the variable $f$, we obtain for $L$ even larger if necessary,
\[
\begin{aligned}
-16L\mathcal Q(t)  \geq & ~{}  \frac12\tilde \sigma \int  |\varphi'| f^2  + 3\tilde\sigma \int  |\varphi'| f_x^2   +2 \tilde\sigma  \int |\varphi'| f_{xx}^2  + \frac14\tilde \sigma\int |\varphi'| f_{xxx}^2 , 
\end{aligned}
\]
%\[
%\sigma > \frac{\al}{3+2\al}, \quad \sigma>\frac{1}{3+\alpha}.
%\]
%Both conditions are equally satisfied if $ \frac{\al}{3+2\al} = \frac{1}{3+\alpha}$, which happens at $\al =\frac12 (-1 + \sqrt{13})\sim 1.303$. This is the choice of $\al$ that we will assume from now on.  Replacing, we have
%\[
%\sigma > \frac{1}{6}(5-\sqrt{13}) =a.
%\]
Then we have
\be\label{Caso2}
\mathcal Q(t)   \lesssim_{\tilde\sigma,L} - \int  |\varphi' |(f^2 +f_x^2 +f_{xx}^2+f_{xxx}^2) \sim -\int |\varphi' | (u_x^2 + u^2). 
\ee
From \eqref{Caso1} and \eqref{Caso2} we conclude that 
\be\label{Caso3}
\mathcal Q(t)   \lesssim  -\int |\varphi' |(u_x^2 + u^2),
\ee
provided $\sigma = -(1+b)$, $b>0$, or $\sigma>\frac18$.  The terms in \eqref{S} can be absorbed by this last term using $L>0$ large and the fact that $|\varphi'''| \lesssim |\varphi'|$. Finally, \eqref{N} can be absorbed by \eqref{Caso3} using \eqref{Smallness} (provided $\ve$ is small enough compared with $b$), just as in \cite{ElDika2,KMPP}. See Appendix \ref{A} for more details. We get
\be\label{Conclu}
\frac{d}{dt}\mathcal H(t)  \lesssim -\int |\varphi' | (u_x^2 + u^2).
\ee
Therefore, we conclude that 
\begin{equation}\label{eq:virial1}
\int_2^{\infty} \int \sech^2 \left(\frac{x +\sigma t}{L}\right) \left(u^2 + u_x^2 \right)(t,x) \, dx\,dt \lesssim_{L}\ve^2.
\end{equation}
This proves \eqref{Conclusion_1}. As an immediate consequence, there exists an increasing sequence of time $t_n \to \infty$ as $n \to \infty$ such that
\begin{equation}\label{eq:virial2}
\int \sech^2 \left(\frac{x+\sigma t_n}{L}\right) \left(u^2 + u_x^2  \right)(t_n,x) \; dx \longrightarrow 0 \mbox{ as } n \to \infty.
\end{equation}
\bigskip

\subsection{End of proof of the Theorem \ref{Thm1}}\label{7}

Consider $\mathcal I(t)$ in \eqref{I}. Choose now $\varphi := \frac12(1+\tanh)$ (for the right side) and $\varphi := \frac12(1-\tanh)$ (for the left hand side) in \eqref{I}. The conclusion \eqref{Conclusion_0} follows directly from the ideas in \cite{MM2}. Indeed,
% \eqref{Virial0}, we have
%\[
%\left|\frac{d}{dt} \mathcal I(t)\right| \lesssim \int \sech^2\left(\frac{x +\sigma t}{L}\right)(u^2 + u_x^2). 
%\]
%Integrating inside the interval $[t, t_n]$, and sending $n$ to infinity (using \eqref{eq:virial2}), we get
%\[
%\mathcal I(t) \lesssim \int_t^\infty \int \sech^2\left(\frac{x +\sigma t}{L}\right)(u^2 + u_x^2)dxdt.
%\]
%Sending $t\to +\infty$, and using \eqref{eq:virial1}, we conclude (passing to well-known argument in \cite{MM2}) \eqref{Conclusion_0}.
%
%\begin{remark}
for the right side (i.e. $((1+b)t, \infty)$, $b > 0$ fixed), we choose $\tilde b = \frac{b}{2}$ and fix $t_0 > 2$. For $2 < t \le t_0$ and large $L \gg 1$ (to make all estimates above hold), we consider the  functional $\mathcal I_{t_0}(t)$ by
\[
\mathcal I_{t_0}(t) := \frac12 \int \varphi \left(\frac{x + \sigma t_0 - \tilde \sigma (t_0 -t)}{L} \right) \left( u^2 + u_x^2\right)(t,x)dx,
\]
where $\sigma = -(1+b)$ and $\tilde \sigma = -(1 + \tilde b)$. From Lemma \ref{Virial_bous}, \eqref{Positivity +} with $\tilde b > 0$ and the smallness condition \eqref{Smallness}, we have
\[\begin{aligned}
\frac{d}{dt} \mathcal I_{t_0}(t) \lesssim_{\tilde b, L} &~  {}  - \int \sech^2 \left(\frac{x + \sigma t_0 - \tilde \sigma (t_0 -t)}{L}\right) (u^2 + u_x^2) \le 0,
\end{aligned}\]
which shows that the new functional $\mathcal I_{t_0} (t)$ is decreasing on $[2,t_0]$. On the other hand, since $\lim_{x \to -\infty} \varphi(x) =0$, we have
\[
\limsup_{t \to \infty}\int \varphi \left(\frac{x -\beta t - \gamma }{L} \right) (u^2 + u_x^2)(\delta, x)dx =0,
\]
for any fixed $\beta, \gamma, \delta >0$. Together with all above, for any $2 < t_0$, we have
\[\begin{aligned}
0 \le&~{} \int \varphi \left(\frac{x - (1+b) t_0}{L} \right) \left( u^2 + u_x^2\right)(t_0,x)dx\le  \int \varphi \left(\frac{x -(b-\tilde b) t_0 - 2(1+ \tilde b)}{L} \right) \left( u^2 + u_x^2\right)(2,x)dx,
\end{aligned}\]
which implies
\[\limsup_{t \to \infty} \int \varphi \left(\frac{x - (1+b) t}{L} \right) \left( u^2 + u_x^2\right)(t,x)dx = 0.\]
A analogous argument can be applied to the left side ($(-\infty, -at)$, fixed $a > \frac18$), thus we conclude \eqref{Conclusion_0}.
%\end{remark}

\begin{remark}
The understanding of the decay procedure inside the interval $(-t/8,t)$ is an interesting open problem that we hope to consider in a forthcoming publication (see \cite{KM2}), at least in the case $p=2$. See also \cite{MP} for similar recent results in the KdV case.
\end{remark}

\appendix

\section{About the proof of \eqref{Conclu}}\label{A}

In this section we estimate the nonlinear term
\[
\begin{aligned}
\mathcal N(t) =& ~{}\frac{\al}2\int \varphi' \left( 2f (1-\partial_x^{2})^{-1}(u^p) + ( (1-\partial_x^{2})^{-1}u^p)^2  \right)  \\
& ~{} - \frac{\al}2\int \varphi' \left( 2f_x (1-\partial_x^{2})^{-1}(u^p)_x + (\partial_x (1-\partial_x^{2})^{-1}u^p)^2 \right)\\
&~ {}     -\frac1{p+1}(\al \sigma -1) \int  \varphi' u^{p+1}     + \int \varphi'  u (1-\partial_x^2)^{-1}\left(u^p\right).
\end{aligned}
\]
Clearly,
\[
\left| \frac1{p+1}(\al \sigma -1) \int  \varphi' u^{p+1} \right| \lesssim  \ve^{p-1}\int  |\varphi' |u^2,
\]
which is enough. Now, recall the following results.
\begin{lemma}[\cite{ElDika2}]\label{Dika1}
The operator $ (1-\px^2)^{-1}$ satisfies the following comparison principle: for any $u,v\in H^1$,
\begin{equation}\label{eq:inverse op1}
v \le w \quad  \Longrightarrow \quad (1-\px^2)^{-1} v \le (1-\px^2)^{-1} w.
\end{equation}
\end{lemma}
Also,
\begin{lemma}[\cite{ElDika2,KMPP}]\label{lem:nonlinear1}
Suppose that $\phi =\phi(x)$ is such that
\begin{equation}\label{eq:inverse op2}
(1-\px^2)^{-1}\phi(x) \lesssim \phi(x), \quad x\in \R,
\end{equation}
for $\phi(x) > 0$ satisfying $|\phi^{(n)}(x)| \lesssim \phi(x)$, $n \ge 0$. Then, for $v,w \in H^1$, we have
%\begin{equation}\label{eq:nonlinear1-1}
%\int \phi^{(n)} v (1-\px^2)^{-1}(wh)_x ~\lesssim ~ \norm{v}_{H^1} \int \phi (w^2 + w_x^2 +h^2 + h_x^2),
%\end{equation}
\begin{equation}\label{eq:nonlinear1-2}
\int\phi^{(n)} v (1-\px^2)^{-1}(w^p) ~\lesssim ~\norm{v}_{H^1}\norm{w}_{H^1}^{p-2} \int \phi w^2
\end{equation}
%\begin{equation}\label{eq:nonlinear2-1}
%\int (\phi v_x)_x (1-\px^2)^{-1}(wh) \lesssim \norm{v}_{H^1} \int \phi (w^2 + w_x^2 +h^2 + h_x^2),
%\end{equation}
and
\begin{equation}\label{eq:nonlinear3-1}
\int \phi v_x (1-\px^2)^{-1} (w^p)_x \lesssim \norm{v}_{H^1}\norm{w}_{H^1}^{p-2} \int \phi (w^2 + w_x^2).
\end{equation}
\end{lemma}

Using \eqref{eq:nonlinear1-2} with $n=0$, 
\[
\left| \int \varphi'  u (1-\partial_x^2)^{-1}\left(u^p\right) \right| \lesssim \ve \int |\varphi' | |u^p| \lesssim \ve^{p-1} \int |\varphi' | u^2.
\]

Using \eqref{eq:nonlinear1-2} with $n=0$ and \eqref{eq:nonlinear3-1},  we also have from $\|f\|_{L^{\infty}}, \|f_x\|_{L^{\infty}} \lesssim \|u\|_{H^1}$ that
\[
 \left| \int \varphi' f (1-\partial_x^{2})^{-1}(u^p) \right| \lesssim \ve  \int |\varphi'| |u^p| \lesssim \ve^{p-1} \int |\varphi' | u^2
\]
and
\[
 \left| \int \varphi' f_x (1-\partial_x^{2})^{-1}(u^p)_x \right| \lesssim \ve  \int |\varphi'| |(u^p)_x| \lesssim \ve^{p-1} \int |\varphi' | (u^2 + u_x^2).
\]
For the rest terms, using \eqref{eq:nonlinear1-2} with $n=0$ and \eqref{eq:nonlinear3-1},
\[
\int |\varphi' | ((1-\partial_x^{2})^{-1}(u^p))^2  \lesssim \| (1-\partial_x^{2})^{-1}(u^p) \|_{H^1} \ve^{p-2} \int  |\varphi' | u^2
\]
and
\[
\int |\varphi' | ((1-\partial_x^{2})^{-1}\partial_x(u^p))^2 \lesssim \| (1-\partial_x^{2})^{-1}(u^p) \|_{H^1} \ve^{p-2} \int  |\varphi' | (u^2 + u_x^2).
\] 
Finally, $ \| (1-\partial_x^{2})^{-1}(u^p) \|_{H^1}  \lesssim \|u^p\|_{H^{-1}} \lesssim \ve^{p}$. Gathering these estimates, we get for some $\delta$ small enough,
\[
\abs{\mathcal N(t)}  ~\lesssim \delta \int  |\varphi' | (u^2 + u_x^{2}).
\]

\providecommand{\bysame}{\leavevmode\hbox to3em{\hrulefill}\thinspace}
\providecommand{\MR}{\relax\ifhmode\unskip\space\fi MR }
\providecommand{\MRhref}[2]{%
  \href{http://www.ams.org/mathscinet-getitem?mr=#1}{#2}
}
\providecommand{\href}[2]{#2}

\end{document}